\documentclass [11pt]{amsart}
\usepackage{amsfonts}
\usepackage{graphicx}
\usepackage{pdfsync}
\usepackage{amscd}
\usepackage{mathdots}
\usepackage{amsmath,amsfonts,amssymb,amsthm,mathrsfs,xypic}
\usepackage{tikz}
\usetikzlibrary{arrows,decorations.pathmorphing,backgrounds,positioning,fit,petri}
\xyoption{curve}
\usepackage{fancybox}
\usepackage{booktabs}
 \usepackage{systeme}
 
\theoremstyle{plain}

\newcommand{\dms}{\operatorname{dim}}

\newcommand{\rad}{\operatorname{rad}}
\newcommand{\soc}{\operatorname{soc}}
\newcommand{\Hom}{\operatorname{Hom}}
\newcommand{\End}{\operatorname{End}}

\newcommand{\Ext}{\operatorname{Ext}}
\newcommand{\Irr}{\operatorname{Irr}}
\newcommand{\ind}{\operatorname{ind}}
\newcommand{\modd}{\operatorname{mod}}

\newcommand{\Ima}{\operatorname{Im}}
\newcommand{\Mesh}{\operatorname{Mesh}}
\newcommand{\bk}{\mathbf{k}}
\newcommand{\gldim}{\operatorname{gldim}}

\newtheorem{thm}{Theorem}[subsection]

\newtheorem{cor}[thm]{Corollary}

\newtheorem{lem}[thm]{Lemma}
\newtheorem{exm}[thm]{Example}

\newtheorem*{results*}{Results}

\newtheorem{prop}[thm]{Proposition}
\newtheorem{rmk}[thm]{Remark}
\newtheorem{defn}[thm]{Definition}
\usepackage{tikz-cd}
\usetikzlibrary{matrix,arrows,decorations.pathmorphing}

\begin{document}
\title[Preprojective algebras of tree-type quivers]{Preprojective algebras of tree-type quivers}



\author{Van C.~Nguyen}
\address{Department of Mathematics\\Northeastern University\\Boston, MA 02115}
\email{v.nguyen@northeastern.edu} 

\author{Gordana Todorov}
\email{g.todorov@northeastern.edu} 

\author{Shijie Zhu}
\email{zhu.shi@husky.neu.edu} 


\maketitle

\begin{abstract}
Let $Q$ be a tree-type quiver, $\bk Q$ its path algebra, and $\lambda$ a nonzero element in the field $\bk$. We construct irreducible morphisms in the Auslander-Reiten quiver of the transjective component of the bounded derived category of $\bk Q$ that satisfy what we call the $\lambda$-relations. When $\lambda=1$, the relations are known as mesh relations. When $\lambda=-1$, they are known as commutativity relations. Using this technique together with the results given by Baer-Geigle-Lenzing, Crawley-Boevey, Ringel, and others, we show that for any tree-type quiver, several descriptions of its preprojective algebra are equivalent. 
\end{abstract}


\section{Introduction}

Let $\bk$ be a field. All considered algebras are over $\bk$ and all modules are left modules. For a $\bk$-algebra $A$, denote by $A$-$\modd$ the category of all finitely presented $A$-modules. 

Let $Q=(Q_0,Q_1)$ be a finite quiver, where $Q_0$ is the set of vertices and $Q_1$ is the set of arrows. Let $\bk Q$ be the path algebra of $Q$ with coefficient field $\bk$. We use the following convention to compose paths: suppose $p$ is a path from $a$ to $b$ and $q$ is a path from $c$ to $d$, then $qp$ is a path from $a$ to $d$ given by concatenation if $b=c$, and $qp=0$ if $b\neq  c$. Let $e_a \in \bk Q$ be the trivial path at vertex $a$.

The preprojective algebra of $Q$ is widely studied in many fields in mathematics. Gelfand and Ponomarev in \cite{GP} were originally interested in constructing an algebra $A$ with the following property (*): $A$ contains $\bk Q$ as a subalgebra and when considered as a left $\bk Q$-module, $A$ decomposes as a direct sum of the indecomposable ``preprojective" $\bk Q$-modules, one from each isomorphism class. Hence, the name ``preprojective algebra" naturally comes from this property. However, it was emphasized by Ringel \cite{R2} that for a fixed quiver $Q$, there may be several isomorphism classes of algebras A with property (*). 

There are several definitions of the preprojective algebra of a given quiver $Q$ in various contexts: it is defined as the quotient $\Pi$ of the path algebra of the double quiver $\overline{Q}$ modulo the relation $\sum_{\alpha \in Q_1} (\alpha^* \alpha-\alpha\alpha^*)$ (c.f.~Definition~\ref{defn:quiver}); it is also defined by Baer-Geigle-Lenzing as the orbit algebra $\Sigma':= \bigoplus_{i\geq0} \Hom_{\bk Q}(\bk Q,\tau^{-i}\bk Q)$, where $\tau^- = TrD$ (c.f.~Definition~\ref{defn:BGL}). Dually one can consider the algebra $\Sigma:= \bigoplus_{i\geq0} \Hom_{D^b(\bk Q)}(\tau^i \bk Q,\bk Q)$ (c.f.~Definition~\ref{defn:sigma}). Furthermore, it is proved by Baer-Geigle-Lenzing in \cite[Proposition 3.1]{BGL} that the preprojective algebra $\Sigma'$ is isomorphic to the $\Ext$-tensor algebra $\bk Q \langle \Omega \rangle$ of a $\bk Q$-bimodule $\Omega:= \Ext^1_{\bk Q}(D(\bk Q), \bk Q)$ (c.f.~Definition~\ref{defn:tensor alg}).

Using Morita equivalence and forgetful functors to $\bk$-$\modd$, Ringel \cite[Theorem A]{R2} proved that if $Q$ is an acyclic quiver (i.e. $Q$ has no oriented cycles), then the preprojective algebra $\Pi$ of $Q$ is isomorphic to the $\Ext$-tensor algebra $\bk Q \langle \Omega \rangle$. On the other hand, using a different approach related to the deformed preprojective algebras introduced by Crawley-Boevey and Holland in \cite{CBH}, Crawley-Boevey gave another proof for $\Pi \simeq \bk Q \langle \Omega \rangle$ in \cite[Theorems 2.3 and 3.1]{CB}. 

In this paper, for any tree-type quiver $Q$ (i.e. the underlying graph of $Q$ is a tree), we construct explicit algebra isomorphisms between $\Pi$ and $\Sigma$ in Theorem~\ref{thm:alg isom}, between $\Pi$ and $\Sigma'$ in Theorem~\ref{thm:alg isom2}, and between $\Sigma$ and $\Sigma'$ in Corollary~\ref{cor:dual}. Our results together with \cite[Proposition 3.1]{BGL} can be combined to give a description of the isomorphism $\Pi \simeq \bk Q \langle \Omega \rangle$ whose existence was already proved independently by Ringel and Crawley-Boevey.  As a consequence, we show that these definitions of the preprojective algebra of a tree-type quiver $Q$ are equivalent. 

\begin{results*}
Let $Q$ be a tree-type quiver. The followings are isomorphic as algebras: 
$$
\xymatrixcolsep{6.5pc} \xymatrix{\Sigma'=\bigoplus_{i\geq0} \Hom_{\bk Q}(\bk Q,\tau^{-i}\bk Q) \ar[r]^{ \ \ \ \ \ \ \ \ \ \ \  \simeq, \text{ \cite{BGL}}} \ar@{<..>}[d]_{\simeq, \text{ Cor.\ref{cor:dual}}}^{\delta} & \bk Q \langle \Omega \rangle \ar[d]^{\simeq, \text{ \cite{CB,R2}}} \\
\Sigma=\bigoplus_{i\geq0} \Hom_{D^b(\bk Q)}(\tau^i \bk Q,\bk Q) \ar@{<..>}^{\simeq, \text{ Thm.\ref{thm:alg isom}}}_{\eta}[r] & \Pi=\bk\overline{Q}/(\sum_{\alpha \in Q_1} (\alpha^* \alpha-\alpha\alpha^*)).
}
$$
\end{results*}

Our paper is organized as follows. In Section~\ref{sec:prelim}, we recall some background material on the Auslander-Reiten theory and the mesh category. 

In Section~\ref {sec:irred}, we examine how one could fill irreducible morphisms in the Auslander-Reiten quiver of the transjective component of the bounded derived category of $\bk Q$ such that those morphisms satisfy certain relations. In \ref{subsec:filling}, for a tree-type quiver $Q$ and a nonzero $\lambda \in \bk$, we construct irreducible morphisms satisfying the so-called $\lambda$-relations. When $\lambda=-1$, these relations are known as the commutativity relations. This plays an important role in defining the isomorphism in main Theorem~\ref{thm:alg isom}. Nonetheless, we also point out that this technique does not work for some non-tree quivers in \ref{subsec:non-tree}. 

In Section~\ref{sec:preprojective}, we focus on various descriptions of the preprojective algebras of $Q$. In \ref{subsec:preprojective}, we give details of such descriptions, in particular, the definitions of the algebras $\Pi, \Sigma', \Sigma$, and $\bk Q \langle \Omega \rangle$. In \ref{subsec:alg isom}, for any tree-type quiver $Q$, we construct explicit algebra isomorphisms $\eta: \Sigma \rightarrow \Pi$ and $\delta: \Sigma \rightarrow \Sigma'$ to complete the above diagram. 


As mentioned above, Gelfand and Ponomarev \cite{GP} gave a construction of preprojective algebras of quivers $Q$ using the relation $\sum_{\alpha \in Q_1} (\alpha^* \alpha-\alpha\alpha^*)$. Dlab and Ringel \cite{DR} extended this construction to preprojective algebras of modulated graphs. On the other hand, it is interesting to note that one can also define an algebra using the relation $\sum_{\alpha \in Q_1} (\alpha^* \alpha+\alpha\alpha^*)$. In general, as pointed out by Gabriel in \cite{G}, this algebra may not be isomorphic to the algebra $\Pi$ defined by quiver and relations as above. However, in the case when $Q$ is a tree-type quiver, they are isomorphic, \cite[\S 6]{R2}. Preprojective algebras with relations defined with plus signs, their deformations and relation to Lusztig algebras, were studied in \cite{MOV} mostly from a geometric point of view. For the results of this paper, preprojective algebras with relations given by minus signs are more relevant.

\vskip10pt
 
\textbf{Acknowledgement:} The authors thank Bill Crawley-Boevey for insightful discussions and for pointing out some relevant results on preprojective algebras.


\section{Preliminaries}
\label{sec:prelim}


\subsection{Almost split sequences}
\label{subsec:almost}

As we are working over both the module category $\bk Q$-$\modd$ and the bounded derived category $D^b(\bk Q)$, whenever we refer to ``almost split sequence,'' it means almost split sequence in the context of abelian category and almost split triangle in the context of triangulated category.  

Due to Auslander and Reiten, we have the following description of almost split sequences starting from an indecomposable projective non-injective module of a hereditary algebra.
\begin{prop} \cite[Proposition 2.4]{AR5} 
\label{ass AR}
Let $H$ be a hereditary algebra and $P$ be an indecomposable projective noninjective $H$-module. Denote $P^*=\Hom_H(P,H)$ which is an $H^{op}$-module and $\rad P$ the radical of $P$. Then we have an almost split sequence 
$$
\xymatrix{0\ar[r]&P\ar[r]&(\rad P^*)^*\oplus\tau^-\rad P\ar[r]&\tau^-P\ar[r]&0},
$$
where $\tau^-=TrD$ is the inverse of the Auslander-Reiten translation, and the morphism 
$$(\rad P^*)^*\longrightarrow \tau^-P/\Ima (\tau^-\rad P\rightarrow \tau^-P)$$ is a projective cover. 
\end{prop}

Suppose $Q=(Q_0,Q_1)$ is an acyclic quiver and $P_j=\bk Qe_j$ is the indecomposable projective $\bk Q$ module corresponding to the vertex $j$. Then $\rad P_j\simeq \mathop{\bigoplus}\limits_{(j\rightarrow i) \in Q_1}P_i$ and $(\rad P_j^*)^*\simeq\mathop{\bigoplus}\limits_{(k\rightarrow j) \in Q_1}P_k$. Hence using Proposition $\ref{ass AR}$,  we have the following description of almost split sequences in $\bk Q$-$\modd$.

\begin{lem}
In $\bk Q$-$\modd$, the almost split sequence starting from any indecomposable projective module $P_j$ is:
$$
\xymatrix{0\ar[r]&P_j\ar[r]&\mathop{\bigoplus}\limits_{(k\rightarrow j) \in Q_1}P_k \oplus \mathop{\bigoplus}\limits_{(j\rightarrow i) \in Q_1} \tau^-P_i\ar[r]& \tau^-P_j\ar[r]&0.}
$$ 
The almost split sequence starting from any indecomposable preprojective module $\tau^{-n}P_j$, for $n\geq 0$, is:
$$
\xymatrix{0\ar[r]&\tau^{-n}P_j\ar[r]&\mathop{\bigoplus}\limits_{(k\rightarrow j) \in Q_1}\tau^{-n}P_k \oplus \mathop{\bigoplus}\limits_{(j\rightarrow i) \in Q_1} \tau^{-(n+1)}P_i\ar[r]& \tau^{-(n+1)}P_j\ar[r]&0.}
$$ 
\end{lem}
Closely related to almost split sequences are \textbf{irreducible morphisms}, which are by definition, noninvertible morphisms which have no proper factorization. Since a morphism  $f: X\rightarrow Y$  (with either $X$ or $Y$ indecomposable) is irreducible  if and only if it can be completed into a left or right minimal almost split morphism \cite[Theorem 2.4]{AR4}, we can describe the irreducible morphisms in the preprojective component of $\bk Q$-$\modd$ as:


\begin{lem}\label{irr}
Let $Q=(Q_0,Q_1)$ be an acyclic quiver. Let $X$ and $Y$ be indecomposable preprojective modules over the path algebra $\bk Q$. Let $f: X\rightarrow Y$ be a non-zero morphism. Then $f$ is an irreducible morphism if and only if one of the following holds: \\
$(1)$ $X\simeq \tau^{-n} P_i$, $Y\simeq \tau^{-n} P_j$, $n\geq0$  and there is an arrow $j\rightarrow i$ in $Q$.\\
$(2)$ $X\simeq \tau^{-n} P_i$, $Y\simeq \tau^{-(n+1)} P_j$, $n\geq0$ and there is and arrow $i\rightarrow j$ in $Q$.
\end{lem}


\subsection{Auslander-Reiten quiver and standard components}
\label{subsec:AR quiver}

We recall here some properties of the mesh category and of the Auslander-Reiten quiver (AR-quiver) for any $\Hom$-finite Krull-Schmidt additive $\bk$-category $\mathcal C$; for more details see for example~\cite{ASS,L,LP,R}. Auslander-Reiten quiver was originally introduced by Bautista, Riedtmann and Ringel. As we work in both categories $\bk Q$-$\modd$ and $D^b(\bk Q)$, we will need the generalized notions of almost split sequences (also called Auslander-Reiten sequences) and the Auslander-Reiten quiver in a Krull-Schmidt category, as introduced by Liu in \cite{L}. This unifies the notion of an almost split sequence in an abelian category and that of an Auslander-Reiten triangle in a triangulated category. 

\begin{defn}
Given a quiver $Q$, its {\bf path category} is an additive category, whose objects are finite direct sums of indecomposable objects. The indecomposable objects are the vertices of $Q$, and given two indecomposable objects $a$ and $b$, the set of morphisms from $a$ to $b$ is the $\bk$-vector space whose basis is the set of all paths from $a$ to $b$.
\end{defn}

Let $\Gamma=(\Gamma_0, \Gamma_1)$ be a \textbf{locally finite quiver}, that is, for every vertex in $\Gamma$, the number of arrows going in and out of that vertex is finite. A {\bf translation quiver} $\Gamma=(\Gamma_0,\Gamma_1, t)$ is a locally finite quiver $(\Gamma_0, \Gamma_1)$  together with an injective map $t: \Gamma'_0\rightarrow \Gamma_0$ defined on a subset $\Gamma_0' \subseteq \Gamma_0$ such that for any $z\in \Gamma'_0$ and any $y\in \Gamma_0$ the number of arrows from $y$ to $z$ is the same as the number of arrows from $t(z)$ to $y$. The vertices in $\Gamma_0 \setminus \Gamma'_0$ are called {\bf projective}.
Given a translation quiver $\Gamma=(\Gamma_0, \Gamma_1, t)$, let $\Gamma_1' := \{\alpha: a\rightarrow b\ |\ b \in \Gamma'_0 \} \subseteq \Gamma_1$. A {\bf polarization} of $\Gamma$ is an injective map $\sigma: \Gamma_1' \rightarrow \Gamma_1$, such that $\sigma (\alpha): t(b) \rightarrow a$. In the path category of $(\Gamma_0, \Gamma_1)$, 
define the {\bf mesh ideal} as the ideal generated by elements
$$
m_z=\mathop{\sum}\limits_{(\alpha:\, x \rightarrow z)}\alpha\sigma(\alpha).
$$

\begin{defn}\label{mesh def}
The {\bf mesh category} $\Mesh(\Gamma, \sigma)$ of a translation quiver $\Gamma$ with a polarization $\sigma$ is defined as the path category of $\Gamma$ modulo the mesh ideal.
\end{defn}

Let $\bk$ be an arbitrary field and $\mathcal C$ be a $\Hom$-finite, Krull-Schmidt additive $\bk$-category. For each indecomposable object $X$ in $\mathcal C$, we define $k_X:=\End(X)/\rad(X,X)$ to be the automorphism field of $X$. For indecomposable objects $X$ and $Y$, we define $\Irr(X,Y):= \rad(X,Y)/\rad^2(X,Y)$, 
 which is generally referred to as the bimodule of irreducible morphisms from $X$ to $Y$.
 Let $d'_{XY}:=\dim_{k_X}\Irr(X,Y)$ and $d_{XY}:=\dim_{k_Y}\Irr(X,Y)$.

We recall here the general definition of the Auslander-Reiten quiver. However, for this paper, we will only consider the non-valued quiver (trivial valuation). 
\begin{defn}\label{AR quiver def}
The {\bf Auslander-Reiten quiver} $\Gamma_{\mathcal C}$ of $\mathcal C$ is defined to be a valued quiver with vertex set being a complete set of the representatives of the isomorphism classes of indecomposable objects in $\mathcal C$. For vertices $X$ and $Y$, we draw a unique valued arrow $X\rightarrow Y$ with valuation $(d_{XY}, d'_{XY})$ if and only if $d_{XY}>0$. 
\end{defn}

It is proved by Liu \cite[Proposition 2.1]{L} that $\Gamma_{\mathcal C}$ is a valued translation quiver with the translation given by: $\tau Z = X$ if and only if $\mathcal C$ has an almost split sequence $X \rightarrow Y \rightarrow Z$. When the valuation of the translation quiver is symmetric, that is $d_{XY}=d_{XY}'$, $\Gamma_{\mathcal C}$ is modified in such a way that each symmetrically valued arrow is replaced by $d_{XY}$ unvalued arrows from $X$ to $Y$. Therefore, in general, when the valuation of the translation quiver is symmetric,  we can define the mesh category as in Definition~\ref{mesh def}. For a $\Hom$-finite, Krull-Schmidt additive $\bk$-category $\mathcal C$, if we assume its Auslander-Reiten quiver $\Gamma_{\mathcal C}$ has a symmetric valuation, we can define the mesh category of $\Gamma_{\mathcal C}$. 

In the case $\mathcal C=\bk Q$-$\modd$, where $Q$ is a tree, we have $d'_{XY}=d_{XY}=1$ for all indecomposable modules $X$ and $Y$. So the valuation is trivial and $\Gamma_{\mathcal C}$ is a non-valued quiver, also called simply laced.

\begin{rmk}
In case quiver $\Gamma$ does not have multiple arrows, there is a unique polarization $\sigma$. In this case, we can denote the mesh category $\Mesh(\Gamma,\sigma)$ briefly as $\Mesh(\Gamma)$. An example of such a quiver is the AR-quiver $\Gamma$ of $\bk Q$-$\modd$, when $Q$ is a tree-type quiver. In this case, the AR-quiver $\Gamma$ of $\bk Q$-$\modd$ is a translation quiver with translation $t$ being the Auslander-Reiten translation $\tau$ and $\Mesh(\Gamma)$ is the corresponding mesh category. 
\end{rmk}

Recall that in our definition of the AR-quiver $\Gamma_\mathcal C$, for a $\Hom$-finite, Krull-Schmidt additive $\bk$-category $\mathcal C$, the vertices of $\Gamma_\mathcal C$ are \emph{representatives} of isomorphism classes of indecomposable objects of $\mathcal C$. So the indecomposable objects of the path category of $\Gamma_\mathcal C$ 
are \emph{representatives} of isomorphism classes of indecomposable objects in $\mathcal C$. We are interested in the case when the mesh category $\Mesh(\Gamma_\mathcal C)$ is equivalent to a subcategory of $\mathcal C$ consisting of objects in $\Mesh(\Gamma_\mathcal C)$. We will see in Theorem~\ref{LP2.2} below that it holds for the subcategory of preprojective and preinjective modules of $\bk Q$-$\modd$ for a strongly locally finite quiver $Q$. Recall that a quiver $Q$ is \textbf{strongly locally finite} if it is locally finite and interval finite (the number of paths between any two given vertices is finite). In particular, any finite quiver is strongly locally finite. We will also see in Theorem ~\ref{LP2.3} that the same result also holds for the transjective component of the bounded derived category $D^b(\bk Q)$.

\begin{defn}
\label{standard}
Let $\Gamma_\mathcal C$ be the AR-quiver for a $\Hom$-finite, Krull-Schmidt additive $\bk$-category $\mathcal C$. Let $\Gamma'$ be a convex subquiver of $\Gamma_\mathcal C$, that is, every path from a vertex in $\Gamma'$ to a vertex in $\Gamma'$ lies in $\Gamma'$. Let $\mathcal C_{\Gamma'}$ be the full subcategory of $\mathcal C$ consisting of objects in the path category of $\Gamma'$. We say that $\Gamma'$ is {\bf standard} provided that every indecomposable object $X$ in $\Gamma'$ has automorphism field $k_X \cong \bk$, and there exists a $\bk$-equivalence $F: \Mesh(\Gamma') \rightarrow \mathcal C_{\Gamma'}$, which acts as the identity on objects, that is $F(X)=X$, for all objects $X$ in $\Mesh(\Gamma')$. 
\end{defn}

\begin{thm} \cite[Theorems 2.2]{LP} 
\label{LP2.2}
Let $Q$ be a connected quiver which is strongly locally finite. Then the preprojective and preinjective components of $\Gamma_{rep^+(Q)}$ are standard.
\end{thm}

From this theorem, we know that in the preprojective component there is a set of irreducible morphisms satisfying the mesh relations. If a morphism $f$ can be written as a sum of compositions of irreducible morphisms from that set, then $f=0$ if and only if the sum of compositions of irreducible morphisms is generated by the mesh relations. Finally, this will be used to show the well-definedness of our homomorphism between two algebras $\Sigma$ and $\Pi$ in the main Theorem~\ref{thm:alg isom}.

 
\subsection{Derived category and Auslander-Reiten triangles}
\label{subsec:derived cat}

Auslander-Reiten theory was generalized to the derived category of the path algebra by Happel \cite{H}. 
 
\begin{defn}
Suppose $H$ is a hereditary algebra. The {\bf transjective component $\mathcal T$} of the bounded derived category $D^b(H)$ is the full subcategory containing preprojective modules and $(-1)$-shifted preinjective modules. 
\end{defn}

The Auslander-Reiten triangles in $D^b(H)$ are described by Happel \cite[\S 5.4]{H}. Recall that there are two kinds of Auslander-Reiten triangles in $D^b(H)$:

$(1)$ If $Z$ is an indecomposable non-projective $H$-module, then there is an almost split sequence in $H$-$\modd$:
$$
\xymatrix{0\ar[r]&X\ar[r]^u&Y\ar[r]^v&Z\ar[r]&0.}
$$
Let $w$ be the element in $\Ext^1_{H}(Z,X)=\Hom_{D^b(H)}(Z,X[1])$ corresponding to this almost split sequence. Then the following are Auslander-Reiten triangles:
$$
\xymatrix{X[i]\ar[r]^{u[i]}&Y[i]\ar[r]^{v[i]}&Z[i]\ar[r]^(0.4){w[i]}&X[i+1], & \text{ for all } i \in \mathbb Z}.
$$

$(2)$ For an indecomposable projective $H$-module $P$, let $I=\nu P$ be the indecomposable injective module corresponding to $P$ under the Nakayama functor $\nu$, that is, $I$ is the injective envelope of $P/\rad P$. Then there are Auslander-Reiten triangles:
$$
\xymatrix{I[i]\ar[r]&(I/{\soc I})[i]\oplus (\rad P)[i+1]\ar[r]&P[i+1]\ar[r]&I[i+1], & \text{ for all } i \in \mathbb Z}.
$$

The following is needed for our technique of filling irreducible morphisms in Section~\ref{subsec:filling}. 
\begin{thm} \cite[Theorems 2.3]{LP} 
\label{LP2.3}
 Let $Q$ be a connected quiver which is strongly locally finite. Then the transjective component of $\Gamma_{D^b(rep^+(Q))}$ is standard.
\end{thm}



\section{Irreducible morphisms in the AR-quiver}
\label{sec:irred}

\subsection{The filling technique}
\label{subsec:filling}

In this section, let $Q$ be a tree-type quiver. We will show the technique of constructing a set of irreducible morphisms in the Auslander-Reiten quiver of the transjective component $\mathcal T$ of $D^b(\bk Q)$ such that they satisfy the so-called $\lambda$-relations, which will later be used to prove our main Theorem~\ref{thm:alg isom}. In particular, we answer the following problem.

\vskip10pt
\noindent
\textbf{Problem:}  Let $Q=(Q_0,Q_1)$ be a tree-type quiver.  Let $\{ f_{ij}: P_i\rightarrow P_j\ |\ (j\rightarrow i) \in Q_1\}$ be a set of irreducible morphisms. Let $\lambda\in \bk^\times$. Is there a set of morphisms $\{g^{(\lambda)}_{ij}\}$, where $g^{(\lambda)}_{ij}: \tau P_j\rightarrow P_i$, such that for each vertex $j\in Q_0$, they satisfy the following {\bf $\lambda$-relations}?
\begin{eqnarray}\sum_{(k\rightarrow j) \in Q_1}g^{(\lambda)}_{jk}\tau f_{jk} + \lambda\sum_{(j\rightarrow i) \in Q_1}f_{ij}g^{(\lambda)}_{ij}=0 \label{lambda relation}\end{eqnarray} 
That is, for each $j$, we have {\bf $\lambda$-commutative squares}:
\begin{center} 
\tiny\begin{tikzpicture}[->]
\node (1) at (-2,0) {$\tau P_j$};
\node (2) at (0,1) {$\mathop{\bigoplus}\limits_{(k\rightarrow j) \in Q_1}\tau P_k$};
\node (3) at (0,-1) {$\mathop{\bigoplus}\limits_{(j\rightarrow i) \in Q_1} P_i$};
\node (4) at (2,0) {$P_j$ .};
\node (5) at(0,0) {$(\lambda)$};
 
\draw (1)--node[above] {$(\tau f_{jk})$}(2) ;
\draw (1)--node[below] {$(g^{(\lambda)}_{ij})$}(3);
\draw (2)--node[above] {$(g^{(\lambda)}_{jk})$}(4);
\draw (3)--node[below] {$(f_{ij})$}(4);
\end{tikzpicture}
\end{center}

\begin{rmk}\label{relations}
$(a)$ When $\lambda=1$, the $(1)$-relations are known as the {\bf mesh relations}:
\begin{eqnarray}\sum_{(k\rightarrow j) \in Q_1}g^{(1)}_{jk}\tau f_{jk}+\sum_{(j\rightarrow i) \in Q_1}f_{ij}g^{(1)}_{ij}=0.\label{mesh relation}\end{eqnarray} 
Later in this section, we will explain a functorial correspondence between these relations and the mesh ideal in Definition $\ref{mesh def}$. Hence, we refer to these as ``mesh relations''.

$(b)$When $\lambda=-1$, the $(-1)$-relations are known as the {\bf commutativity relations}: 
\begin{eqnarray}\sum_{(k\rightarrow j) \in Q_1}g^{(-1)}_{jk}\tau f_{jk}-\sum_{(j\rightarrow i) \in Q_1}f_{ij}g^{(-1)}_{ij}=0.\label{commutativity}\end{eqnarray} 

$(c)$ The motivation of $\lambda$-relations comes from \cite[\S 6]{R2}, where the author considers the algebra $\Pi_\lambda :=\bk{\overline Q}/\mathop{\sum}\limits_{\alpha\in Q_1}(\alpha^*\alpha-\lambda\alpha\alpha^*)$.
\end{rmk}

\vskip10pt
\noindent
\textbf{Our solution:}  (see Theorem \ref{main thm})
Our strategy of constructing such morphisms $\{g^{(\lambda)}_{ij}\}$ is divided into two steps.

\noindent {\bf Step 1.} Let $\{ f_{ij}: P_i\rightarrow P_j\ |\ (j\rightarrow i) \in Q_1\}$ be the given set of  irreducible morphisms. By the existence theorem of almost split sequences, starting from a source point of $Q$, we can successively find irreducible morphisms
$u_{ij}: \tau P_i\rightarrow \tau P_j$ and $v_{ij}: \tau P_j\rightarrow P_i$, such that for each $j\in Q_0$,  they satisfy the mesh relations (via a similar correspondence with the mesh ideal, as in Remark $\ref{relations}$ (a)):
\begin{eqnarray}
\sum_{(k\rightarrow j) \in Q_1} v_{jk}u_{jk}+\sum_{(j\rightarrow i) \in Q_1}f_{ij}v_{ij}=0, \label{mesh}
\end{eqnarray}
that is, we have $1$-commutative squares:
\begin{center} 
\tiny\begin{tikzpicture}[->]
\node(1) at (-2,0) {$\tau P_j$};
\node (2) at (0,1) {$\mathop{\bigoplus}\limits_{(k\rightarrow j) \in Q_1}\tau P_k$};
\node (3) at (0,-1) {$\mathop{\bigoplus}\limits_{(j\rightarrow i) \in Q_1} P_i$};
\node (4) at (2,0) {$ P_j$ .};
\node (5) at(0,0) {$(\lambda=1)$};

\draw (1)--node[above] {$(u_{jk})$\ \ }(2) ;
\draw (1)--node[below] {$(v_{ij})$}(3);
\draw (2)--node[above] {$(v_{jk})$}(4);
\draw (3)--node[below] {$(f_{ij})$}(4);
\end{tikzpicture}
\end{center}

\noindent {\bf Step 2.}
Given a fixed choice of the set of irreducible morphisms
$\{u_{ij}, v_{ij}\}$ as in Step 1, we construct $\{g^{(\lambda)}_{ij}\}$ by scalar modification of $\{v_{ij}\}$. We will describe this construction below, and provide a detailed illustration in Example~\ref{example}.

Notice that the morphisms $u_{ij}$ are not uniquely determined when we complete the almost split sequences. On one hand, $u_{ij}$ are irreducible morphisms from $\tau P_i$ to $\tau P_j$; on the other hand, since $\tau$ is a functor on $\mathcal T$, we know that $\tau f_{ij}$ are also irreducible morphisms from $\tau P_i$ to $\tau P_j$. In general, not all $u_{ij}$ can be chosen to be precisely $\tau f_{ij}$, however:
 
\begin{lem}
\label{cij lem}
For each arrow $(j\rightarrow i)$, there exists a $c_{ij}\in \bk^\times$ such that $c_{ij}\tau f_{ij} = u_{ij}.$ 
 \end{lem}
 \begin{proof}
It is clear that $\tau f_{ij}$ and $u_{ij}$ belong to  $\Irr(\tau P_i,\tau P_j)$. Since $Q$ does not have multiple arrows, $\dim_\bk\Irr(\tau P_i,\tau P_j)=1$. Hence for some $c_{ij}\in\bk^\times$,  $c_{ij}\tau f_{ij} = u_{ij}$.
 \end{proof}
 
Observe that the coefficients $c_{ij}$ are associated with the arrows $(j\rightarrow i)$ in the quiver. We will use these coefficients $c_{ij}$, obtained in Lemma~\ref{cij lem}, to later define a function $\varphi_\lambda$ from a set of walks to $\bk^{\times}$, which we will use in the scalar modification process to construct irreducible morphisms $\{g^{(\lambda)}_{ij}\}$.

In a quiver, a reverse arrow of an arrow $\alpha$ is formally denoted by $\alpha^{-1}$. It reverses the starting and ending vertex of $\alpha$. 
A {\bf walk} $w$ in the quiver is a finite sequence $w_nw_{n-1}\cdots w_1$, where $w_i=\alpha^{\pm}$ is either an arrow or reverse arrow, and $w_{i+1}\neq w_i^{-1}$ for all $i$. Denote $W$ to be the set of all walks on the quiver $Q$.

\begin{lem}
Suppose $Q$ is a tree-type quiver and $a$, $b$ are vertices in $Q$. Then there is a unique walk from $a$ to  $b$.
\end{lem}

\begin{proof} 
Because $Q$ is a tree-type quiver, a walk $w$ can only pass a vertex at most once. Hence it is unique. 
\end{proof}

Denote by $w(a,b)$ the unique walk from $a$ to $b$, and by $|w|$ the length of the walk $w$.
%
We define an evaluation function on the walks as follows:

%
%
%


\begin{defn}
For a tree-type quiver $Q$ and some $\lambda\in\bk^\times$, define $\varphi_\lambda:W\rightarrow \bk^\times$  by $\varphi_\lambda(\alpha)=\lambda c_{ij}$  for an arrow $(\alpha: j\rightarrow i)$, and extend it to all the walks by $\varphi_\lambda(w_2w_1)=\varphi_\lambda(w_2)\varphi_\lambda(w_1)$ if $w_2w_1$ is composable, and $\varphi_\lambda(w^{-1})=\varphi_\lambda(w)^{-1}$.
 \end{defn}

\begin{lem}
\label{phi function}
Let $Q$ be a tree type quiver and $w$ be a walk starting at vertex $j$ in $Q$.
\begin{enumerate}
\item Let $\alpha$ be an arrow $k\rightarrow j$ and $\beta$ be an arrow $j\rightarrow k$. Then 
$\varphi_\lambda(w\alpha)= \lambda c_{jk}\varphi_\lambda(w)$ and $\varphi_\lambda(w\beta^{-1})=\lambda^{-1} c_{jk}^{-1}\varphi_\lambda(w)$.
\item Function $\varphi_{-1}(w)=(-1)^{|w|}\varphi_{1}(w)$, where $|w|$ is the length of the walk $w$.
\end{enumerate}
\end{lem}
\begin{proof} (1) follows from definition, (2) follows by induction on the length of the walks.
\end{proof}

The maps $\varphi_{\lambda}$ will be used to modify  irreducible maps $v_{ij}$ from Step 1. in order to get irreducible maps $g^{(\lambda)}_{ij}$ in Step 2. which will satisfy $\lambda$-relations.
\begin{lem}
\label{main lem}
Suppose $Q$ is a tree-type quiver and $\{ f_{ij}: P_i\rightarrow P_j\ |\ (j\rightarrow i) \in Q_1\}$ is a set of irreducible morphisms. Then there exist irreducible morphisms $\{g^{(\lambda)}_{ij}\}_{(j \rightarrow i)\in Q_1}$ satisfying the $\lambda$-relations for each $j\in Q_0$:
$$\sum_{(k\rightarrow j) \in Q_1}g^{(\lambda)}_{jk}\tau f_{jk}+\lambda\sum_{(j\rightarrow i) \in Q_1}f_{ij}g^{(\lambda)}_{ij}=0.$$
\end{lem}

\begin{proof}
 Label an arbitrary sink vertex  by $1$. Let $w(j,1)$ be the unique walk from any vertex $j$ to sink vertex $1$. Let $\{ f_{ij}: P_i\rightarrow P_j\ |\ (j\rightarrow i) \in Q_1\}$ be the given set of irreducible morphisms. Let $\{u_{ij},v_{ij}\}$ be the irreducible morphisms in $\mathcal T$ satisfying the mesh relations 
 $(\ref{mesh})$
$$\sum_{(k\rightarrow j) \in Q_1} v_{jk}u_{jk}+\sum_{(j\rightarrow i) \in Q_1}f_{ij}v_{ij}=0.
$$
For each vertex $j \in Q_0$,  define $g^{(\lambda)}_{ij} := \varphi_\lambda (w(j,1))v_{ij}$. 
Next, show that there is an identity: 
$$
\mathop{\sum}\limits_{(k\rightarrow j)}g^{(\lambda)}_{jk}\tau f_{jk}+\lambda\mathop{\sum}\limits_{(j\rightarrow i)} f_{ij}g^{(\lambda)}_{ij}=\lambda\varphi_\lambda(w(j,1)) \left(\mathop{\sum}\limits_{(k\rightarrow j)}v_{jk}u_{jk}+\mathop{\sum}\limits_{(j\rightarrow i)} f_{ij}v_{ij} \right).
$$
In fact, we only need to check the mesh relation ending in $P_j$ for each vertex $j$.\\
Since $g^{(\lambda)}_{ij}=\varphi_\lambda(w(j,1))v_{ij}$ and by Lemma $\ref{cij lem}$ we have $u_{ij}=c_{ij}\tau f_{ij}$, it follows that:
$$
\mathop{\sum}\limits_{(k\rightarrow j)}g^{(\lambda)}_{jk}\tau f_{jk}+\lambda\mathop{\sum}\limits_{(j\rightarrow i)} f_{ij}g^{(\lambda)}_{ij}
=\mathop{\sum}\limits_{(k\rightarrow j)}\varphi_\lambda(w(k,1))v_{jk}c^{-1}_{jk}u_{jk}+\lambda \mathop{\sum}\limits_{(j\rightarrow i)} f_{ij}\varphi_\lambda(w(j,1))v_{ij}.$$
However, since there is an arrow $(k\rightarrow j)$, the walk $w(k,1)$ is the concatenation of the walk $w(j,1)$ by the arrow $(k\rightarrow j)$. Due to Lemma \ref{phi function}, $\varphi_\lambda(w(k,1))=\lambda c_{jk}\varphi_\lambda(w(j,1))$. Plugging into the previous equality, we have
\begin{eqnarray*}
\mathop{\sum}\limits_{(k\rightarrow j)}g^{(\lambda)}_{jk}\tau f_{jk}+\lambda\mathop{\sum}\limits_{(j\rightarrow i)} f_{ij}g^{(\lambda)}_{ij}
&=&\mathop{\sum}\limits_{(k\rightarrow j)}\varphi_\lambda(w(k,1))v_{jk}c^{-1}_{jk}u_{jk}+\lambda \mathop{\sum}\limits_{(j\rightarrow i)} f_{ij}\varphi_\lambda(w(j,1))v_{ij}\\
 &=&\lambda\mathop{\sum}\limits_{(k\rightarrow j)}\varphi_\lambda(w(j,1))v_{jk}u_{jk}+\lambda \mathop{\sum}\limits_{(j\rightarrow i)} f_{ij}\varphi_\lambda(w(j,1))v_{ij}\\
 &=&\lambda\varphi_\lambda(w(j,1)) \left(\mathop{\sum}\limits_{(k\rightarrow j)}v_{jk}u_{jk}+\mathop{\sum}\limits_{(j\rightarrow i)} f_{ij}v_{ij} \right).
\end{eqnarray*} \vspace{-1em}
Therefore $\mathop{\sum}\limits_{(k\rightarrow j) \in Q_1}g^{(\lambda)}_{jk}\tau f_{jk}+\lambda\mathop{\sum}\limits_{(j\rightarrow i) \in Q_1}f_{ij}g^{(\lambda)}_{ij}=0$.

\end{proof}

\begin{thm}
\label{main thm}
Suppose $Q$ is a tree-type quiver and $\{ f_{ij}: P_i\rightarrow P_j\ |\ (j\rightarrow i) \in Q_1\}$ is a set of irreducible morphisms. 
\begin{enumerate}
\item There exist irreducible morphisms $\{g^{(\lambda)}_{ij}\}_{(j \rightarrow i)\in Q_1}$ satisfying 
\begin{eqnarray}\sum_{(k\rightarrow j) \in Q_1}\tau^ng^{(\lambda)}_{jk}\tau^{n+1} f_{jk}+\lambda\sum_{(j\rightarrow i) \in Q_1}\tau^nf_{ij}\tau^ng^{(\lambda)}_{ij}=0.\end{eqnarray} 
 for each $j\in Q_0$ and $n\in \mathbb Z$.
 \item There exist irreducible morphisms $\{g^{(\lambda)}_{ij}\}_{(j \rightarrow i)\in Q_1}$ such that the AR-quiver of $\mathcal T$ can be filled with irreducible morphisms $\tau^nf_{ij}$ and $\tau^ng^{(\lambda)}_{ij}$, 
satisfying the $\lambda$-relations. 
\end{enumerate}
\end{thm}

\begin{proof}
Part $(1)$, by Lemma $\ref{main lem}$, there exist irreducible morphisms $\{g^{(\lambda)}_{ij}\}$ such that for each $j$, they satisfy the $\lambda$-relations
$$\sum_{(k\rightarrow j) \in Q_1}g^{(\lambda)}_{jk}\tau f_{jk}+\lambda\sum_{(j\rightarrow i) \in Q_1}f_{ij}g^{(\lambda)}_{ij}=0.$$
Assign each arrow $(\tau^n P_i\rightarrow \tau^n P_j)$ an irreducible morphisms $\tau^nf_{ij}$ and each arrow $(\tau^{n+1}P_j\rightarrow \tau^nP_i)$ an irreducible morphisms $\tau^ng^{(\lambda)}_{ij}$. Since $\tau$ is an equivalence functor on $\mathcal T$, we have 
$$\sum_{(k\rightarrow j) \in Q_1}\tau^ng^{(\lambda)}_{jk}\tau^{n+1} f_{jk}+\lambda\sum_{(j\rightarrow i) \in Q_1}\tau^nf_{ij}\tau^ng^{(\lambda)}_{ij}=0.
$$
Part $(2)$ is straightforward.
\end{proof}
%
%
\begin{cor}
\label{sign diff 2}
$(a)$ There exist irreducible morphisms $\{g^{(1)}_{ij}\}_{(j \rightarrow i)\in Q_1}$ such that the AR-quiver of $\mathcal T$ can be filled with irreducible morphisms $\tau^nf_{ij}$ and $\tau^ng^{(1)}_{ij}$, for $n\in \mathbb Z$, satisfying the mesh relations for each $j\in Q_0$,
\begin{eqnarray}\sum_{(k\rightarrow j) \in Q_1}\tau^ng^{(1)}_{jk}\tau^{n+1} f_{jk}+\sum_{(j\rightarrow i) \in Q_1}\tau^nf_{ij}\tau^ng^{(1)}_{ij}=0.\label{mesh relation 2}\end{eqnarray} 
$(b)$  There exist irreducible morphisms $\{g^{(-1)}_{ij}\}_{(j \rightarrow i)\in Q_1}$ such that the AR-quiver of $\mathcal T$ can be filled with irreducible morphisms $\tau^nf_{ij}$ and $\tau^ng^{(-1)}_{ij}$, for $n\in\mathbb Z$, satisfying the commutativity relations for  each $j\in Q_0$,
\begin{eqnarray}\sum_{(k\rightarrow j) \in Q_1}\tau^ng^{(-1)}_{jk}\tau^{n+1} f_{jk}-\sum_{(j\rightarrow i) \in Q_1}\tau^nf_{ij}\tau^ng^{(-1)}_{ij}=0.\label{comm2}\end{eqnarray} 
$(c)$ It follows that $\tau^ng^{(1)}_{ij}=(-1)^{|w(j,1)|} \tau^ng^{(-1)}_{ij}$, where $w(j,1)$ is the unique walk from any vertex $j$ to a fixed sink vertex $1$.
\end{cor}
\begin{proof} $(a)$ and $(b)$ are special cases of Theorem \ref{main thm}, when $\lambda=1$ or $\lambda=-1$.\\
$(c)$ $g^{(1)}_{ij}=\varphi_1(w(j,1))v_{ij}=(-1)^{|w(j,1)|}\varphi_{-1}(w(j,1))v_{ij}$ by definition and Lemma \ref{phi function}.
Also by definition $g^{(-1)}_{ij}=\varphi_{-1}(w(j,1))v_{ij}$, and after applying $\tau^n$, the statement $(c)$ follows.
\end{proof}
\begin{defn}
\label{hg}
We denote the set of irreducible morphisms obtained in Corollary $\ref{sign diff 2} \, (a)$ by $\widetilde{\mathcal H}=\{\tau^n f_{ij}, \tau^n g^{(1)}_{ij} \}$ and the set of irreducible morphisms obtained in Corollary $\ref{sign diff 2} \, (b)$ by $\mathcal H=\{\tau^n f_{ij}, \tau^n g^{(-1)}_{ij}\}$.
\end{defn}

Since $\dms_\bk \Irr(M,N)=1$, for all indecomposable objects $M, N \in \mathcal T$, we observe that $\widetilde{\mathcal H}$ or $\mathcal H$ exhausts all the irreducible morphisms up to scalar multiplications. 

Moreover, recall that $\mathcal T$ is standard by Theorem \ref{LP2.3}. By Definition~$\ref{standard}$, $\mathcal T$ is standard given the existence of a functor $F$ between the mesh category $\Mesh(\mathcal T)$ and the full subcategory $\ind \mathcal T$ of $\mathcal T$ consisting only the representatives of indecomposable objects in $\mathcal T$. Indeed, the set $\widetilde{\mathcal H}$ gives rise to this functor $F$: 
by the proof of \cite[Theorems 1.4]{LP}, if we assign each arrow in the AR-quiver with an irreducible morphism such that they satisfy the mesh relations, then we can define a $\bk$-equivalence functor $F: \Mesh(\mathcal T) \rightarrow \mathcal \ind\, \mathcal T$ by 
$$
F(\tau^n P_i\rightarrow \tau^n P_j)=\tau^n f_{ij} \quad \text{ and } \quad F(\tau^{n+1}P_j\rightarrow \tau^nP_i )=\tau^n g^{(1)}_{ij}.
$$
Note that by definition, the mesh ideal which defines $\Mesh(\mathcal T)$ is given by 
$$
m_{\tau^n P_j}=\mathop{\sum}\limits_{(\alpha:\, \tau^n P_i \rightarrow \tau^n P_j)}\alpha\sigma(\alpha)+\mathop{\sum}\limits_{(\beta:\, \tau^{n+1} P_k \rightarrow \tau^n P_j)}\beta\sigma(\beta).
$$
One can check that under $F$, the mesh ideal $m_{\tau^nP_j}$ corresponds to the mesh relation $(\ref{mesh relation 2})$.

We now prove an important result which is used in the proof of main Theorem \ref{thm:alg isom}. 

\begin{lem}\label{well def}
Let $Q$ be a tree-type quiver and $\mathcal T$ be the transjective component of $D^b(\bk Q)$. 
Any morphism $f$ in $\mathcal T$ can be written as a sum of compositions of irreducible morphisms $h_{t,i} \in \mathcal H$: 
$$f=\mathop{\sum}\limits_{t}a_t  h_{t,m_t}\circ\cdots\circ  h_{t,1}.$$ 
Then $f=0$ if and only if $\mathop{\sum}\limits_{t}a_t  h_{t,m_t}\circ\cdots\circ h_{t,1}$ is generated by commutativity relations $(\ref{comm2})$.
\end{lem}

\begin{proof}
Let $\widetilde{\mathcal H}$ and $\mathcal H$ be sets of irreducible morphisms obtained simultaneously as in 
Corollary~$\ref{sign diff 2}$ and defined in Definition $\ref{hg}$.
%

Assume $f$ can be written as a sum of compositions of irreducible morphisms 
$$f=\mathop{\sum}\limits_{t}c_t  \tilde h_{t,m_t}\circ\cdots\circ  \tilde h_{t,1},$$ 
where $ \tilde h_{t,i}$ is in the set $\widetilde{\mathcal H}$. 
 Then $f=0$ if and only if $\mathop{\sum}\limits_{t}c_t  \tilde h_{t,m_t}\circ\cdots\circ  \tilde h_{t,1}$  satisfies the mesh relations $(\ref {mesh relation 2})$. That is, $f=\mathop{\sum}\limits_{t}c_t  \tilde h_{t,m_t}\circ\cdots\circ  \tilde h_{t,1}$ is generated by $$\mathop{\sum}\limits_{(k\rightarrow j) \in Q_1}\tau^ng^{(1)}_{jk}\tau^{n+1} f_{jk}+\mathop{\sum}\limits_{(j\rightarrow i) \in Q_1}\tau^nf_{ij}\tau^ng^{(1)}_{ij}.$$

However, according to Corollary $\ref{sign diff 2} \, (c)$,
 \begin{eqnarray*}
 &&\mathop{\sum}\limits_{(k\rightarrow j) \in Q_1}\tau^ng^{(1)}_{jk}\tau^{n+1} f_{jk}+\mathop{\sum}\limits_{(j\rightarrow i) \in Q_1}\tau^nf_{ij}\tau^ng^{(1)}_{ij}\\
 &=&\mathop{\sum}\limits_{(k\rightarrow j) \in Q_1} \left((-1)^{|w(k,1)|} \ \tau^ng^{(-1)}_{jk}\tau^{n+1} f_{jk} \right) + (-1)^{|w(j,1)|}\mathop{\sum}\limits_{(j\rightarrow i) \in Q_1}\tau^nf_{ij}\tau^ng^{(-1)}_{ij}\\
 &=&(-1)^{|w(j,1)|+1} \left(\mathop{\sum}\limits_{(k\rightarrow j) \in Q_1}\tau^ng^{(-1)}_{jk}\tau^{n+1} f_{jk}-\mathop{\sum}\limits_{(j\rightarrow i) \in Q_1}\tau^nf_{ij}\tau^ng^{(-1)}_{ij} \right).
 \end{eqnarray*}
Hence, $f=0$ if and only if $f$ is generated by $\mathop{\sum}\limits_{(k\rightarrow j) \in Q_1}\tau^ng^{(-1)}_{jk}\tau^{n+1} f_{jk} \ - \mathop{\sum}\limits_{(j\rightarrow i) \in Q_1}\tau^nf_{ij}\tau^ng^{(-1)}_{ij}$, which are the commutativity relations $(\ref{comm2})$.
 \end{proof}

Theorem \ref{main thm} is stated in a constructive way, and we will explain how to construct the coefficients $\varphi_{-1}(w(j,1))$ step by step in Example \ref{example}. Before that, we need the following result:
\begin{lem}\label{modify}
Suppose we have the following exact sequence:
$$
\xymatrix{0\ar[r]&A\ar[rr]^{(f_1, f_2, \cdots, f_n)^T}&&B\ar[rr]^{(g_1,g_2,\cdots, g_n)}&&C\ar[r]&0.}
$$
Then for $c_i, c\in \bk^\times$ invertible, we have the following exact sequences:

$(a) \quad
\xymatrix @C=2.2pc {0\ar[r]&A\ar[rrr]^{(c_1^{-1}f_1, c_2^{-1}f_2, \cdots, c_n^{-1}f_n)^T}&&&B\ar[rrr]^{(c_1g_1,c_2g_2,\cdots, c_ng_n)}&&&C\ar[r]&0}
$

$(b) \quad
\xymatrix @C=2pc {0\ar[r]&A\ar[rr]^{(f_1, f_2, \cdots,  cf_n)^T}&&B\ar[rrr]^{(cg_1,cg_2,\cdots, cg_{n-1}, g_n)}&&&C\ar[r]&0}
$
 
$(c) \quad
\xymatrix @C=2pc {0\ar[r]&A\ar[rr]^{(cf_1, cf_2, \cdots,  cf_n)^T}&&B\ar[rrr]^{(g_1,g_2,\cdots, g_{n-1}, g_n)}&&&C\ar[r]&0}.
$
\end{lem}

\begin{proof}
Here we will just prove part $(a)$, the other parts follow by similar argument. The isomorphism between two complexes:
$$
\xymatrix @R=1.05pc @C=2.4pc {0\ar[r]&A\ar@{=}[ddd]\ar[rrr]^{(f_1, f_2, \cdots, f_n)^T}&&&B\ar[ddd]^{\left(\begin{smallmatrix}
c_1^{-1}&&0\\
&\ddots&\\
0&&c_n^{-1}
\end{smallmatrix}\right)
} \ar[rrr]^{(g_1,g_2,\cdots, g_n)}&&&C\ar[r]\ar@{=}[ddd]&0\\
\\
\\
0\ar[r]&A\ar[rrr]^{(c_1^{-1}f_1, c_2^{-1}f_2, \cdots, c_n^{-1}f_n)^T}&&&B\ar[rrr]^{(c_1g_1,c_2g_2,\cdots, c_ng_n)}&&&C\ar[r]&0}
$$
induces the isomorphism between the corresponding homologies. So the upper sequence is exact if and only if so is the lower sequence.
\end{proof}

\begin{exm}
\label{example}
\normalfont
Suppose $Q$ is the following quiver:
$$
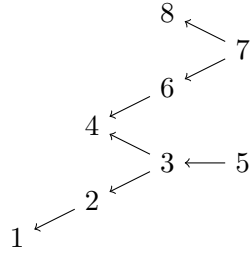
\begin{tikzpicture}[->]
\node(1) at (-1,-0.5) {$1$};
\node (2) at (0,0) {$2$};
\node (5) at (2,0.5) {$5$};
\node (4) at (0,1) {$4$};
\node (3) at (1, 0.5) {$3$};
\node (6) at (1,1.5) {$6$};
\node (7) at (2,2) {$7$};
\node (8) at (1,2.5) {$8$};

\draw (2)--(1);
\draw (3)--(2);
\draw (3)--(4);
\draw (5)--(3);
\draw (6)--(4);
\draw (7)--(6);
\draw (7)--(8);
\end{tikzpicture}
$$

Suppose we have filled the translation quiver with the irreducible morphisms $\{u_{ij}, v_{ij} \}$ which are constructed in Section~$\ref{subsec:filling}$. Remember that we have $c_{ij}\tau f_{ij}=u_{ij}$ by Lemma ~\ref{cij lem}.
$$
\tiny{\begin{tikzpicture}[->]
\node(1) at (-1,-1) {$P_1$};
\node (2) at (0,-0.5) {$P_2$};
\node (5) at (2,0) {$P_5$};
\node (4) at (0,.5) {$P_4$};
\node (3) at (1, 0) {$P_3$};
\node (6) at (1,1) {$P_6$};
\node (7) at (2,1.5) {$P_7$};
\node (8) at (1,2) {$P_8$};

\draw (1)--node[below] {\ \ $f_{12}$}(2) ;
\draw (2)--(3);
\draw (4)--(3);
\draw (3)--(5);
\draw (4)--(6);
\draw (6)--(7);
\draw (8)--(7);

\node(t1) at (-3,-1) {$\tau P_1$};
\node (t2) at (-2,-0.5) {$\tau P_2$};
\node (t5) at (0,0) {$\tau P_5$};
\node (t4) at (-2,.5) {$\tau P_4$};
\node (t3) at (-1, 0) {$\tau P_3$};
\node (t6) at (-1,1) {$\tau P_6$};
\node (t7) at (0,1.5) {$\tau P_7$};
\node (t8) at (-1,2) {$\tau P_8$};

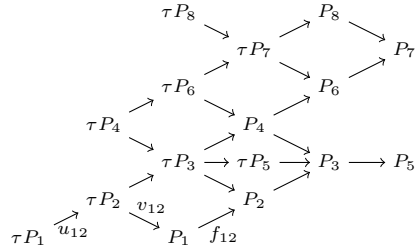
\draw (t1)--node[below]{\ \ $u_{12}$}(t2);  
\draw (t2)--(t3);
\draw (t4)--(t3);
\draw (t3)--(t5);
\draw (t4)--(t6);
\draw (t6)--(t7);
\draw (t8)--(t7);

\draw (t2)--node[above]{$\ \ v_{12}$}(1); 
\draw (t3)--(2);
\draw (t3)--(4);
\draw (t5)--(3);
\draw (t6)--(4); 
\draw (t7)--(6);
\draw (t7)--(8);
\end{tikzpicture}}
$$

We start filling irreducible morphisms from $\tau P_1$. First, we have an exact sequence:
$$
\xymatrix{0\ar[r]&\tau P_1\ar[rr]^{u_{12}=c_{12}\tau f_{12}\ \ }&&\tau P_2\ar[rr]^(0.55){v_{12}}&& P_1\ar[r]&0}.
$$
By Lemma~$\ref{modify} \, (a)$, we have an exact sequence
$$
\xymatrix{0\ar[r]&\tau P_1\ar[rr]^{\tau f_{12}\ \ }&&\tau P_2\ar[rr]^(0.55){c_{12}v_{12}}&& P_1\ar[r]&0}.
$$


Let $g^{(-1)}_{12}:=c_{12}v_{12}$, the bottom square involving $P_1$ satisfies the commutativity relation.
$$
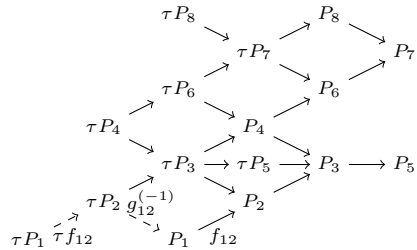
\tiny\begin{tikzpicture}[->]
\node(1) at (-1,-1) {$P_1$};
\node (2) at (0,-0.5) {$P_2$};
\node (5) at (2,0) {$P_5$};
\node (4) at (0,.5) {$P_4$};
\node (3) at (1, 0) {$P_3$};
\node (6) at (1,1) {$P_6$};
\node (7) at (2,1.5) {$P_7$};
\node (8) at (1,2) {$P_8$};

\draw (1)--node[below] {\ \ $f_{12}$}(2) ;
\draw (2)--(3);
\draw (4)--(3);
\draw (3)--(5);
\draw (4)--(6);
\draw (6)--(7);
\draw (8)--(7);

\node(t1) at (-3,-1) {$\tau P_1$};
\node (t2) at (-2,-0.5) {$\tau P_2$};
\node (t5) at (0,0) {$\tau P_5$};
\node (t4) at (-2,.5) {$\tau P_4$};
\node (t3) at (-1, 0) {$\tau P_3$};
\node (t6) at (-1,1) {$\tau P_6$};
\node (t7) at (0,1.5) {$\tau P_7$};
\node (t8) at (-1,2) {$\tau P_8$};

\draw[dashed] (t1)--node[below]{\ \ $\tau f_{12}$}(t2);  
\draw (t2)--(t3);
\draw (t4)--(t3);
\draw (t3)--(t5);
\draw (t4)--(t6);
\draw (t6)--(t7);
\draw (t8)--(t7);

\draw[dashed] (t2)--node[above]{$\ \ g^{(-1)}_{12}$}(1); 
\draw (t3)--(2);
\draw (t3)--(4);
\draw (t5)--(3);
\draw (t6)--(4); 
\draw (t7)--(6);
\draw (t7)--(8);
\end{tikzpicture}
$$
Now, since we have exact sequence by Step 1
$$
\xymatrix{0\ar[r]& \tau P_2\ar[rr]^{(u_{23}, v_{12})^T}&& \tau P_3\oplus P_1\ar[rr]^{\ \ \ (v_{23}, f_{12})} &&P_2\ar[r]&0}
$$
that is,
$$
\xymatrix{0\ar[r]& \tau P_2\ar[rr]^{(c_{23}\tau f_{23}, v_{12})^T}&& \tau P_3\oplus P_1\ar[rr]^{\ \ \ (v_{23}, f_{12})} &&P_2\ar[r]&0}.
$$
By Lemma~$\ref{modify} \, (c)$, we have an exact sequence
$$
\xymatrix{0\ar[r]& \tau P_2\ar[rrr]^{(c_{12}c_{23}\tau f_{23}, c_{12}v_{12})^T}&&& \tau P_3\oplus P_1\ar[rr]^{\ \ \ (v_{23}, f_{12})} &&P_2\ar[r]&0}.
$$
By Lemma~$\ref{modify} \, (a)$, we have an exact sequence
$$
\xymatrix{0\ar[r]& \tau P_2\ar[rrr]^{(\tau f_{23}, c_{12}v_{12})^T}&&& \tau P_3\oplus P_1\ar[rr]^{\ \ \ (c_{12}c_{23}v_{23}, f_{12})} &&P_2\ar[r]&0},
$$
where $c_{12}v_{12}=g^{(-1)}_{12}$ as we constructed before. 
Let $g^{(-1)}_{23}:=-c_{12}c_{23}v_{23}$. The second square involving $P_2$ satisfies the commutativity relation. 
$$
\tiny\begin{tikzpicture}[->]
\node(1) at (-1,-1) {$P_1$};
\node (2) at (0,-0.5) {$P_2$};
\node (5) at (2,0) {$P_5$};
\node (4) at (0,.5) {$P_4$};
\node (3) at (1, 0) {$P_3$};
\node (6) at (1,1) {$P_6$};
\node (7) at (2,1.5) {$P_7$};
\node (8) at (1,2) {$P_8$};

\draw[dashed] (1)--node[below] {\ \ $f_{12}$}(2) ;
\draw (2)--(3);
\draw (4)--(3);
\draw (3)--(5);
\draw (4)--(6);
\draw (6)--(7);
\draw (8)--(7);

\node(t1) at (-3,-1) {$\tau P_1$};
\node (t2) at (-2,-0.5) {$\tau P_2$};
\node (t5) at (0,0) {$\tau P_5$};
\node (t4) at (-2,.5) {$\tau P_4$};
\node (t3) at (-1, 0) {$\tau P_3$};
\node (t6) at (-1,1) {$\tau P_6$};
\node (t7) at (0,1.5) {$\tau P_7$};
\node (t8) at (-1,2) {$\tau P_8$};

\draw[dashed] (t1)--(t2);  
\draw[dashed] (t2)--node[above]{$\tau f_{23}\ \ \ \ $}(t3);
\draw (t4)--(t3);
\draw (t3)--(t5);
\draw (t4)--(t6);
\draw (t6)--(t7);
\draw (t8)--(t7);

\draw[dashed] (t2)--node[above]{$\ \ g^{(-1)}_{12}$}(1); 
\draw [dashed] (t3)--node[below]{$g^{(-1)}_{23}\ \ $}(2);
\draw (t3)--(4);
\draw (t5)--(3);
\draw (t6)--(4); 
\draw (t7)--(6);
\draw (t7)--(8);
\end{tikzpicture}
$$

Next we have an exact sequence
$$
\xymatrix{0\ar[r]& \tau P_3\ar[rr]^(0.4){(v_{43}, u_{35}, v_{23})^T}&& P_4\oplus \tau P_5\oplus P_2\ar[rr]^{\ \ \ (f_{43}, v_{35}, f_{23} )} &&P_3\ar[r]&0}
$$
that is,
$$
\xymatrix{0\ar[r]& \tau P_3\ar[rr]^(0.4){(v_{43}, c_{35}\tau f_{35}, v_{23})^T}&& P_4\oplus \tau P_5\oplus P_2\ar[rr]^{\ \ \ \ (f_{43}, v_{35}, f_{23} )} &&P_3\ar[r]&0}.
$$
By Lemma~$\ref{modify} \, (c)$, we have an exact sequence
$$
\xymatrix{0\ar[r]& \tau P_3\ar[rrrrr]^(0.45){(-c_{12}c_{23}v_{43}, -c_{12}c_{23}c_{35}\tau f_{35}, -c_{12}c_{23}v_{23})^T}&&&&& P_4\oplus \tau P_5\oplus P_2\ar[rr]^{\ \ \ \ (f_{43}, v_{35}, f_{23} )} &&P_3\ar[r]&0}.
$$
By Lemma~$\ref{modify} \, (a)$, we have an exact sequence
$$
\xymatrix{0\ar[r]& \tau P_3\ar[rrrr]^(0.4){(-c_{12}c_{23}v_{43}, \tau f_{35}, -c_{12}c_{23}v_{23})^T}&&&& P_4\oplus \tau P_5\oplus P_2\ar[rrr]^{\ \ \ \ (f_{43}, -c_{12}c_{23}c_{35}v_{35}, f_{23} )} &&&P_3\ar[r]&0},
$$
where $c_{12}c_{23}v_{23}=g^{(-1)}_{23}$ by our previous construction. 
Let $g^{(-1)}_{43}:=c_{12}c_{23}v_{43}$,  and let $g^{(-1)}_{35}:=-c_{35}c_{12}c_{23}v_{35}$. The next square involving $P_3$ satisfies the commutativity relation. Notice that $f_{35}g^{(-1)}_{35}=0$. So the commutativity relation at the square involving $P_5$ is also satisfied.
$$
\tiny\begin{tikzpicture}[->]
\node(1) at (-1,-1) {$P_1$};
\node (2) at (0,-0.5) {$P_2$};
\node (5) at (2,0) {$P_5$};
\node (4) at (0,1) {$P_4$};
\node (3) at (1, 0) {$P_3$};
\node (6) at (1,1.5) {$P_6$};
\node (7) at (2,2) {$P_7$};
\node (8) at (1,2.5) {$P_8$};

\draw[dashed] (1)--(2) ;
\draw[dashed] (2)--node[below]{$f_{23}$}(3);
\draw[dashed] (4)--node[above]{$f_{43}$}(3);
\draw[dashed] (3)--node[above]{$f_{35}$}(5);
\draw (4)--(6);
\draw (6)--(7);
\draw (8)--(7);

\node(t1) at (-3,-1) {$\tau P_1$};
\node (t2) at (-2,-0.5) {$\tau P_2$};
\node (t5) at (0,0) {$\tau P_5$};
\node (t4) at (-2,1) {$\tau P_4$};
\node (t3) at (-1, 0) {$\tau P_3$};
\node (t6) at (-1,1.5) {$\tau P_6$};
\node (t7) at (0,2) {$\tau P_7$};
\node (t8) at (-1,2.5) {$\tau P_8$};

\draw[dashed] (t1)--(t2);  
\draw[dashed] (t2)--(t3);
\draw (t4)--(t3);
\draw [dashed](t3)--node[above]{$\tau f_{35}$}(t5);
\draw (t4)--(t6);
\draw (t6)--(t7);
\draw (t8)--(t7);

\draw[dashed] (t2)--(1); 
\draw [dashed] (t3)--node[below]{$g^{(-1)}_{23}\ \ $}(2);
\draw[dashed] (t3)--node[above]{$g^{(-1)}_{43}$}(4);
\draw [dashed] (t5)--node[above]{$g^{(-1)}_{35}$} (3);
\draw (t6)--(4); 
\draw (t7)--(6);
\draw (t7)--(8);
\end{tikzpicture}
$$
Continue this procedure and we can fill the irreducible morphisms as desired, such that they satisfy commutativity relations.

\end{exm}


\subsection{A non-tree-type case}
\label{subsec:non-tree}
As we mentioned before, for some non-tree-type quivers, we may not be able to find irreducible morphisms $\{f_{ij}:P_i\rightarrow P_j\}$ and $\{g_{ij}^{(1)}:\tau P_j\rightarrow P_i\}$ such that $\{\tau^n f_{ij}, \tau^n g_{ij}^{(1)}\}$ satisfy the mesh relations for all $n$. We illustrate an example here:

\begin{exm}\label{ex:non-tree} \normalfont
Suppose $\bk$ is an algebraically closed field with char $\bk\neq2$. Let $Q$ be:
$$
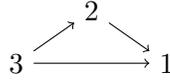
\begin{tikzpicture}[->]
\node (3) at (0,0) {3};
\node (2) at (1,0.7) {2};
\node (1) at (2,0) {1};

\draw (3)--(2);
\draw(2)--(1);
\draw(3)--(1);
\end{tikzpicture}$$

The starting part of the Auslander-Reiten quiver is given by the following picture. \\

\hspace{0.3in}
$
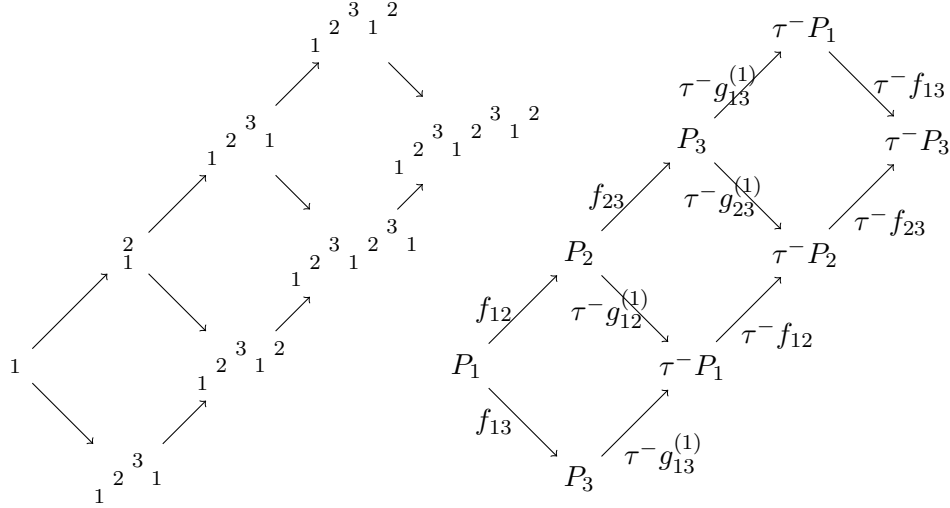
\begin{tikzpicture}[->]
\node(P1) at (-2,1.5) {$\begin{smallmatrix} 1 \end{smallmatrix}$};
\node(P2) at (-.5,3) {$\begin{smallmatrix} 2\\1 \end{smallmatrix}$};
\node(P3) at(1,4.5) {$\begin{smallmatrix} &&3&\\&2&&1\\1  \end{smallmatrix}$};
\node(tP1a) at (2.5,6) {$\begin{smallmatrix} &&3&&2\\&2&&1\\1  \end{smallmatrix}$};
\node(P3A) at (-.5,0) {$\begin{smallmatrix} &&3&\\&2&&1\\1  \end{smallmatrix}$};
\node(tP1) at (1,1.5) {$\begin{smallmatrix} &&3&&2\\&2&&1\\1  \end{smallmatrix}$};
\node(tP2) at (2.5,3) {$\begin{smallmatrix} &&&&&3\\&&3&&2&&1\\&2&&1\\1  \end{smallmatrix}$};
\node(tP3) at (4,4.5) {$\begin{smallmatrix} &&&&&3&&2\\&&3&&2&&1\\&2&&1\\1  \end{smallmatrix}$};

\draw (P1)--(P2);
\draw (P1)--(P3A);
\draw (P2)--(tP1);
\draw (P2)--(P3);
\draw (P3)--(tP1a);
\draw (P3)--(tP2);
\draw (P3A)--(tP1);
\draw (tP1)--(tP2);
\draw (tP2)--(tP3);
\draw (tP1a)--(tP3);

\begin{scope}[transform canvas={xshift = 6cm}]
\node(P1) at (-2,1.5) {$P_1$};
\node(P2) at (-.5,3)  {$P_2$};
\node(P3) at(1,4.5) {$P_3$};
\node(tP1a) at (2.5,6) {$\tau^-P_1$};
\node(P3A) at (-.5,0){$P_3$};
\node(tP1) at (1,1.5){$\tau^-P_1$};
\node(tP2) at (2.5,3)  {$\tau^-P_2$};
\node(tP3) at (4,4.5)  {$\tau^-P_3$};

\draw (P1)--node {$f_{12}\ \ \ \ \ \ $}(P2);
\draw (P1)--node {$f_{13}\ \ \ \ \ \ $}(P3A);
\draw (P2)--node {$\tau^-g^{(1)}_{12}\ \ \ \ \ $}(tP1);
\draw (P2)--node {$f_{23}\ \ \ \ \ \ $}(P3);
\draw (P3)--node {$\tau^-g^{(1)}_{13}\ \ \ \ \ \ $}(tP1a);
\draw (P3)--node {$\tau^-g^{(1)}_{23}\ \ \ \ \ $}(tP2);
\draw (P3A)--node[below] {$\ \ \ \ \ \  \tau^-g^{(1)}_{13}$}(tP1);
\draw (tP1)--node[below] {$\ \ \ \ \ \ \tau^- f_{12}$}(tP2);
\draw (tP2)--node[below] {$\ \ \ \ \ \ \tau^- f_{23}$}(tP3);
\draw (tP1a)--node {$\ \ \ \ \ \ \ \  \ \ \tau^- f_{13}$}(tP3);
\end{scope}
\end{tikzpicture}
$

For the quiver $Q$ in this example, it is easy to see that the irreducible morphisms in the preprojective component of $\bk Q$-mod are monomorphisms. If we choose irreducible morphisms $f_{ij}: P_i\rightarrow P_j=a_{ij}\epsilon_{ij}$, where $a_{ij}\in \bk^\times$ and $\epsilon_{ij}:P_i\rightarrow P_j$ are the canonical embeddings, then one can verify that $\tau^-f_{ij}=a_{ij}\tau^-\epsilon_{ij}$  and $\tau^-\epsilon_{ij}:\tau^-P_i\rightarrow \tau^-P_j$ are the canonical embeddings.


Now suppose we can find $\{g^{(1)}_{ij}\}$ such that the mesh relations $(\ref{mesh relation 2})$ are satisfied. Assume $\tau^-g^{(1)}_{ij}=b_{ij}\iota_{ij}$, where $b_{ij}\in \bk^\times$ and $\iota_{ij}: P_j\rightarrow \tau^- P_i$ is the canonical embedding. Then the mesh relation $(\ref{mesh relation 2})$ is equivalent to the following system of equations:
$$
\begin{cases}
a_{12}b_{12}+a_{13}b_{13}=0\\
a_{23}b_{23}+ a_{12}b_{12}=0\\
a_{13}b_{13}+ a_{23}b_{23}=0,
\end{cases}
$$
which has solutions $a_{12}b_{12}=a_{23}b_{23}=a_{13}b_{13}=0$. This contradicts with the assumptions $a_{ij}\in \bk^\times$ and $b_{ij}\in \bk^\times$.
\end{exm}

\section{Preprojective algebras of tree-type quivers}
\label{sec:preprojective}

In this section, we first recall different descriptions of preprojective algebras of any acyclic quiver $Q$. We then construct algebra isomorphisms to show equivalences between these descriptions when $Q$ is a tree-type quiver. 


\subsection{Definitions of the preprojective algebra}
\label{subsec:preprojective}

The study of preprojective algebras was started by Gelfand and Ponomarev \cite{GP} and developed by Baer-Geigle-Lenzing \cite{BGL} and Ringel \cite{R2}. There are two traditional ways of describing the preprojective algebra (see Definitions $\ref{defn:quiver}$ and $\ref{defn:BGL}$), both of which are proved by Ringel \cite{R2}, Crawley-Boevey \cite{CB} and Baer-Geigle-Lenzing \cite{BGL} to be equivalent to the one given in Definition $\ref{defn:tensor alg}$. We will consider one more description in Definition $\ref{defn:sigma}$. All four descriptions will be proved equivalent for tree-type quivers in the next section. 

\begin{defn}\cite{GP}
\label{defn:quiver}
Let $Q=(Q_0,Q_1)$ be an acyclic quiver. Define the double quiver $\overline{Q}=(\overline{Q}_0, \overline{Q}_1)$ as follows: $\overline{Q}_0=Q_0$ and for each arrow $(\alpha: i\rightarrow j) \in Q_1$, define a reversed arrow $(\alpha^*: j\rightarrow i)$, let $\overline{Q}_1=\{\alpha,\alpha^*\mid \alpha \in Q_1\}$. The {\bf preprojective algebra} of $Q$ is $\Pi:= \bk\overline{Q}/ (\rho)$, defined as a quotient algebra of the path algebra $\bk\overline{Q}$ modulo the relation $\rho=\mathop\sum\limits_{\alpha \in Q_1} (\alpha^*\alpha-\alpha\alpha^*)$.
\end{defn}

There is another definition of preprojective algebra given by Baer-Geigle-Lenzing \cite{BGL}:

\begin{defn}\cite{BGL}
\label{defn:BGL}
The {\bf preprojective algebra} of $Q$ is $\Sigma' :=\mathop\bigoplus\limits_{i\geq0} \Hom_{\bk Q}(\bk Q,\tau^{-i}\bk Q)$, where multiplication is given by $u* v=(\tau^{-t}u)\circ v$, for elements $u\in \Hom_{\bk Q}(\bk Q,\tau^{-s}\bk Q)$ and $v\in \Hom_{\bk Q}(\bk Q,\tau^{-t}\bk Q)$.
\end{defn}

In the same paper, Baer-Geigle-Lenzing \cite[Proposition 3.1]{BGL} proved that preprojective algebra $\Sigma'$ is isomorphic to a certain tensor algebra of some bimodule, as follows:

\begin{defn}
Let $\Lambda$ be a ring and $\Theta$ be a $\Lambda$-bimodule. Let $\Lambda \langle\Theta\rangle$ denote the corresponding tensor algebra which is the direct sum
$$
\Lambda \langle\Theta\rangle=\mathop\bigoplus\limits_{t\geq0}\Theta^{ \otimes t},
$$
where $\Theta^{ \otimes 0}=\Lambda$ by convention. In particular, the multiplication of an element $a\in \Theta^{\otimes s}$ with another element $b\in \Theta^{\otimes t}$ is $a \otimes b \in \Theta^{\otimes (s+t)}=\Theta^{ \otimes s}\,\otimes \,\Theta^{\otimes t}$ if both $s,t \geq 1$, and $a \otimes b=ab$ scalar multiplication otherwise.
\end{defn}


Let $Q$ be an acyclic quiver and $\bk Q$ be its path algebra. Denote $D(\bk Q)=\Hom_\bk(\bk Q,\bk)$ the usual $\bk$-dual. Then $\Omega:= \Ext^1_{\bk Q}(D(\bk Q),\bk Q)$ is a bimodule over $\bk Q$. 
The Auslander-Reiten translations $\tau = DTr$ and $\tau^- = TrD$ on the category $\bk Q$-$\modd$ of all finite dimensional $\bk Q$-modules are well-defined functors since $\bk Q$ is hereditary. 

\begin{defn}
\label{defn:tensor alg}
Let $\Omega:= \Ext^1_{\bk Q}(D(\bk Q),\bk Q)$. The {\bf $\Ext$-tensor algebra} associated to $\bk Q$ is defined to be the tensor algebra $\bk Q\langle\Omega\rangle$. 
\end{defn}

The following lemma is well-known (e.g.~\cite{R2}), we state its proof here for completeness. 

\begin{lem} \label{tau fun}
Let $H$ be any hereditary algebra, then there are equivalences of functors on $H$-$\modd$:
$$
\tau \simeq D\Ext^1_H(-,H) \quad \text{ and } \quad \tau^-\simeq \Ext_H^1(D(H),-).
$$
\end{lem}

\begin{proof} 
Since $\gldim H=1$, for any module $M$, we have a minimal projective resolution:
$$
\xymatrix{0\ar[r]&P_1\ar[r]&P_0\ar[r]&M\ar[r]&0.}
$$

By the definition of transpose $Tr$, we have the following exact sequences and commutative diagram:
$$
\xymatrix{0\ar[r]& \Hom_H(M,H) \ar@{=}[d]\ar[r]& \Hom_H(P_0,H)\ar[r]\ar@{=}[d] &\Hom_H(P_1,H)\ar[r] \ar@{=}[d]&Tr M\ar[r]\ar[d]^{\simeq}&0\\
0\ar[r]& \Hom_H(M,H) \ar[r]& \Hom_H(P_0,H)\ar[r] &\Hom_H(P_1,H)\ar[r] &\Ext_H^1(M,H)\ar[r]&0.
}
$$
Hence, the Auslander-Reiten translation on $M$ is $\tau M=DTr\, M\simeq D\Ext_H^1(M,H)$. It is easy to check that this isomorphism is functorial and we have $\tau\simeq D\Ext_H^1(-,H)$. Similarly, $\tau^-\simeq \Ext_H^1(D(H),-)$.
\end{proof}

\begin{lem}\cite[\S 3]{BGL} 
\label{bgl s3}
Suppose $H$ is a hereditary algebra and $X$ is an $H$-bimodule. Then there is an $H$-bimodule isomorphism:
\begin{eqnarray*}
\Hom_H(H,X)\otimes_H\Hom_H(H,\tau^{-s}H)&\rightarrow& \Hom_H(H,\tau^{-s}X), \text{ by sending }\\
u\otimes v&\mapsto& (\tau^{-s}u)\circ v.
\end{eqnarray*}

\end{lem}

\begin{lem}\cite[\S 3]{BGL}
\label{lem:tensor isom}
Let $\Theta=\Ext_H^1(D(H),H)$. There is an $H$-bimodule isomorphism:
$$
\Theta^{\otimes t}\simeq \Hom_H(H,\tau^{-t}H).
$$
\end{lem}
\begin{proof}
One can easily construct an isomorphism by Lemmas $\ref{tau fun}$, $\ref{bgl s3}$ and induction.
\end{proof}

Hence, from these lemmas, one can construct the following algebra isomorphism given by Baer-Geigle-Lenzing \cite[Proposition 3.1]{BGL} and obtain an equivalent description of the preprojective algebra of $Q$ using the $\Ext$-tensor algebra $\bk Q\langle\Omega\rangle$:

\begin{prop}\cite[Proposition 3.1]{BGL} 
\label{bgl 3.1}
Let $Q$ be an acyclic quiver and $\bk Q\langle\Omega\rangle$ be the $\Ext$-tensor algebra given in Definition $\ref{defn:tensor alg}$. Then there is an algebra isomorphism:
$$\bk Q\langle\Omega\rangle\simeq \Sigma'.$$
 \end{prop}

As a consequence, we observe that $\bk Q\langle\Omega\rangle$ as an $\bk Q$-module is the direct sum of all the indecomposable preprojective $\bk Q$-modules, each occurring with multiplicity one:
$$\bk Q\langle\Omega\rangle \simeq  \Sigma'=\mathop\bigoplus\limits_{i\geq0} \Hom_{\bk Q}(\bk Q,\tau^{-i}\bk Q) \simeq \mathop\bigoplus\limits_{i\geq0}  \tau^{-i}\bk Q.$$
Hence, we can also call $\bk Q\langle\Omega\rangle$ the preprojective algebra. 


%

It is known that $\Pi\simeq \bk Q\langle\Omega\rangle$ as algebras, which was proved independently by Ringel \cite[Theorem A]{R2} and Crawley-Boevey \cite[Theorems 2.3 and 3.1]{CB}. Thus, $\Pi \simeq \Sigma'$ as algebras. This is used to describe finitely presented covariant functors over $\mathcal T$ as modules over preprojective algebra $\Pi$, \cite[Proposition 3.6]{BGL}. In order to study finitely presented contravariant functors as modules over $\Pi$, we need to start with a similarly defined algebra:

\begin{defn}
\label{defn:sigma}
Suppose $H$ is a hereditary algebra. Let $\mathcal T$ be the transjective component of the bounded derived category $D^b(H)$. Define algebra $\Sigma:= \bigoplus_{i\geq 0} \Hom_{\mathcal T}(\tau^i H, H)$. For elements $u\in \Hom_{\mathcal T}(\tau^sH,H)$, $v\in \Hom_{\mathcal T}(\tau^tH,H)$, define multiplication as $u \times v=v \circ \tau^t u$.
\end{defn}

We will construct an explicit algebra isomorphism between the preprojective algebra $\Pi=\bk\overline{Q}/(\sum_{\alpha \in Q_1} (\alpha^* \alpha-\alpha\alpha^*))$ and $\Sigma= \bigoplus_{i\geq 0} \Hom_{\mathcal T}(\tau^i \bk Q, \bk Q)$ in the next section.


\subsection{The algebra isomorphisms}
\label{subsec:alg isom}

Let $Q$ be a tree-type quiver, $\bk Q$ be its path algebra, $\mathcal T$ be the transjective component of the bounded derived category $D^b(\bk Q)$, and $\Pi$, $\Sigma$, $\Sigma'$ be algebras defined as in Section~\ref{subsec:preprojective}. Recall that by the filling technique from Section~\ref{subsec:filling}, we can choose irreducible morphisms $f_{ij}: P_i \rightarrow P_j$ and $g^{(-1)}_{ij}: \tau P_j \rightarrow P_i$ in $\mathcal T$ such that they satisfy the commutativity relations $(\ref{comm2})$ (Corollary \ref{sign diff 2}). Moreover, each irreducible morphisms in $\mathcal T$ is a scalar multiple of an element in $\mathcal H=\{\tau^n f_{ij}, \tau^n g^{(-1)}_{ij}\ |\ n\in\mathbb Z \}$.


\begin{thm}
\label{thm:alg isom}
Suppose $Q$ is a tree-type quiver.  Then there is an algebra isomorphism  
\begin{eqnarray*}
\eta: &\Sigma= \bigoplus_{i\geq 0} \Hom_{\mathcal T}(\tau^i \bk Q, \bk Q) &\longrightarrow \Pi=\bk\overline{Q}/(\sum_{\alpha \in Q_1} (\alpha^* \alpha-\alpha\alpha^*)), \text{ sending} \\
&(f_{ij}: P_i \rightarrow P_j) & \mapsto \text{ arrow } (\alpha: j\rightarrow i), \text{and } \\  
&(g^{(-1)}_{ij}: \tau P_j \rightarrow P_i) & \mapsto \text{ arrow } (\alpha^*: i \rightarrow j),
\end{eqnarray*}
where $\{f_{ij}, \,g^{(-1)}_{ij}\}$ are irreducible morphisms chosen in Section $\ref{subsec:filling}$ such that elements in $\mathcal H=\{\tau^n f_{ij}, \tau^n g^{(-1)}_{ij}\ |\ n\in\mathbb Z \}$ satisfy the commutativity relations.\end{thm}  

\begin{proof}   

(1) {\bf Homomorphism $\eta$}: Define $\eta(f_{ij})= (\alpha: j\rightarrow i)$ and $\eta(g^{(-1)}_{ij})= (\alpha^*: i \rightarrow j)$. Then for each morphism $f\in \Sigma$, we can define $\eta(f)$ in the following way: $f$ can be written as a linear sum of compositions of irreducible morphisms in $\mathcal H$, i.e. $f=\mathop{\sum}\limits_{t}a_t h^{(t)}_{m_t}\circ\cdots\circ h^{(t)}_{1} $, where $a_t\in \bk$ and $h^{(t)}_{s}$ are irreducible morphisms in $\mathcal H$. Now assume the codomain of $h^{(t)}_s$ is $\tau^{k^{(t)}_s}P$ for some indecomposable projective $P$. Then, the composition can be written via the multiplication of $\Sigma$ given in Definition~\ref{defn:sigma} as
$$
f=\mathop{\sum}\limits_{t}a_t (\tau^{-k^{(t)}_1}h^{(t)}_{1}) \times \cdots \times (\tau^{-k^{(t)}_{m_t}}h^{(t)}_{m_t}),
$$
where the last one $k^{(t)}_{m_t}=0$ always. Thus, define $\eta(f)= \mathop{\sum}\limits_{t}a_t \eta(\tau^{-k^{(t)}_1}h^{(t)}_{1})\cdots \eta(\tau^{-k^{(t)}_{m_t}}h^{(t)}_{m_t})$ to be a linear sum of paths.

(2) {\bf $\eta$ is well-defined}: By Lemma~\ref{well def}, if morphism $f=\mathop{\sum}\limits_{t}a_t h^{(t)}_{m_t}\circ\cdots\circ h^{(t)}_{1} =0$, then $f$ is generated by the commutativity relations. So $\eta(f)= \mathop{\sum}\limits_{t}a_t \eta(\tau^{-k^{(t)}_1}h^{(t)}_{1})\cdots \eta(\tau^{-k^{(t)}_{m_t}}h^{(t)}_{m_t})$ is an element in $\Pi$ and generated by the commutativity relations. Consequently, $\eta(f)=0$. 

(3) {\bf Inverse of $\eta$}: To show the injectivity and surjectivity of $\eta$, we directly construct the inverse $\eta^{-1}: \Pi \rightarrow \Sigma$ of $\eta$ as follows. 

For an arrow $(\alpha:j\rightarrow i) \in \overline{Q}_1$ define $\eta^{-1}(\alpha)=f_{ij}$. For an arrow $(\alpha^*:i\rightarrow j) \in \overline{Q}_1$ define $\eta^{-1}(\alpha^*)=g_{ij}^{(-1)}$. For a path $\alpha_1\cdots\alpha_l$, where $\alpha_i\in \overline{Q}_1$, define $\eta^{-1}(\alpha_1\cdots\alpha_l)=\eta^{-1}(\alpha_1)\times\cdots\times\eta^{-1}(\alpha_l)$. Similarly, $\eta^{-1}$ is well defined by the ``if part'' of Lemma~\ref{well def}. It is easy to check that $\eta^{-1}$ is an inverse of $\eta$.
\end{proof}

We can use similar argument for the algebra $\Sigma'$ and get the following dual statement:

\begin{thm}
\label{thm:alg isom2}
Suppose $Q$ is a tree-type quiver. Then there is an algebra isomorphism  
\begin{eqnarray*}
\eta': &\Sigma'=\bigoplus_{i\geq 0} \Hom_{\mathcal T}(\bk Q, \tau^{-i} \bk Q) &\longrightarrow \Pi=\bk\overline{Q}/(\sum_{\alpha \in Q_1} (\alpha^* \alpha-\alpha\alpha^*)), \text{ sending} \\
&(\tau^-g^{(-1)}_{ij}: P_j \rightarrow \tau^-P_i) & \mapsto \text{ arrow } (\alpha: j\rightarrow i), \text{and } \\  
&(f_{ij}: \tau P_i \rightarrow P_j) & \mapsto \text{ arrow } (\alpha^*: i \rightarrow j).
\end{eqnarray*}
\end{thm}

Combine these two isomorphisms, we have the following algebra isomorphism:

\begin{cor}
\label{cor:dual}
Let $Q$ be a tree-type quiver. Let $\Sigma$ and $\Sigma'$ be the algebras defined as before. Then we have the following algebra isomorphism:
\begin{eqnarray*}
\delta: \Sigma= \bigoplus_{i\geq 0} \Hom_{\mathcal T}(\tau^i \bk Q, \bk Q) &\longrightarrow& \Sigma'=\bigoplus_{i\geq 0} \Hom_{\mathcal T}(\bk Q, \tau^{-i} \bk Q), \text{ sending} \\
(id_{P_i})&\mapsto& (id_{P_i})\\
(f_{ij}: P_i \rightarrow P_j)  &\mapsto&(\tau^- g^{(-1)}_{ij}: P_j\rightarrow \tau^-P_i), \text{ and}  \\  
(g^{(-1)}_{ij}: \tau P_j \rightarrow P_i) & \mapsto& (f_{ij}: P_i \rightarrow P_j).
\end{eqnarray*}
\end{cor}  
 
\noindent 
So we have 4 equivalent descriptions of the preprojective algebra of a tree-type quiver $Q$: 

\begin{cor}
\label{cor:equivalent}
For a tree-type quiver $Q$, let $\bk Q$ and $\mathcal T$ be as before. Then the followings are isomorphic as algebras: 
\begin{enumerate}
 \item preprojective algebra $\Pi= \bk\overline{Q}/ \left(\mathop\sum\limits_{\alpha \in Q_1} (\alpha^*\alpha-\alpha\alpha^*) \right)$
 \item $\Ext$-tensor algebra $\bk Q\langle\Omega\rangle$, where $\Omega := \Ext^1_{\bk Q}(D(\bk Q),\bk Q)$ 
 \item orbit algebra $\Sigma':=\displaystyle \bigoplus_{i\geq0} \Hom_{\bk Q}(\bk Q,\tau^{-i}\bk Q)$
 \item $\Sigma:= \displaystyle \bigoplus_{i\geq 0} \Hom_{\mathcal T}(\tau^i \bk Q, \bk Q)$.
\end{enumerate}
\end{cor}

\begin{exm} \normalfont
Let $Q$ be a quiver of type $A_3$:
$\xymatrix{1&2\ar[l]_\alpha&3\ar[l]_\beta}.$ Then its preprojective algebra $\Pi$ is
$
\xymatrix{1\ar@/_/[r]_{\alpha^*}&2\ar@/_/[l]_\alpha\ar@/_/[r]_{\beta^*}&3\ar@/_/[l]_{\beta}}
$,
with relations $\alpha\alpha^*=0$, $\beta^*\beta=0$, and $\alpha^*\alpha=\beta\beta^*$. 

\noindent In order to describe the algebra isomorphisms $\eta, \eta'$ and $\delta$, we need to label the irreducible morphisms in the AR-quiver of $Q$ in the following way:
$$
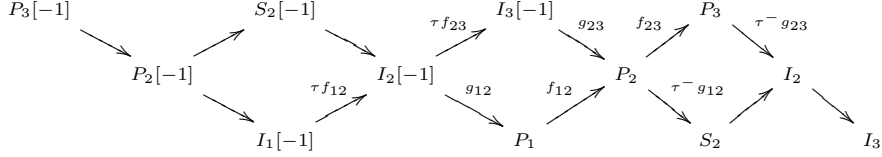
\xymatrixrowsep{10pt}
\xymatrixcolsep{15pt}
\tiny\xymatrix
{
P_3[-1]\ar[dr]&&S_2[-1]\ar[dr]&&I_3[-1]\ar[dr]^{g_{23}}&&P_3\ar[dr]^{\tau^-g_{23}}\\
&P_2[-1]\ar[ur]\ar[dr]&&I_2[-1]\ar[ur]^{\tau f_{23}}\ar[dr]^{g_{12}}&&P_2\ar[ur]^{f_{23}}\ar[dr]^{\tau^-g_{12}}&&I_2\ar[dr]\\
&&I_1[-1]\ar[ur]^{\tau f_{12}}&&P_1\ar[ur]^{f_{12}}&&S_2\ar[ur]&&I_3
 }
$$


Because the preprojective algebra of $Q$ is finite dimensional, isomorphisms $\eta: \Sigma \rightarrow \Pi$, $\eta': \Sigma' \rightarrow \Pi$, and $\delta: \Sigma \rightarrow \Sigma'$ give corresponding bijections in the following table. These isomorphisms send idempotents to idempotents, preserving the left projective modules.

\begin{center}
\begin{tabular}{c|p{1.6cm}|p{6.5cm}|p{2.5cm}}
 & $\ \ \ \Pi e_1$ & $\begin{matrix} &&\ \ \ \ \ \ \ \ \ \ \  \ \ \ \Pi e_2\end{matrix}$ &$\ \ \ \ \Pi e_3$\\
\hline

$\Pi$ & $\begin{matrix} e_1\\ \alpha^* \\ \beta^*\alpha^*\end{matrix}$ &$\begin{matrix}&&&&&e_2\\ \alpha&&&&&&&&&\beta^*\\ &&&&&\alpha^*\alpha=\beta\beta^*\end{matrix}$ &$\begin{matrix}&&e_3\\ &&\beta\\ &&\alpha\beta\end{matrix}$\\
\hline

$\Sigma$ &$\begin{matrix} id_{P_1}\\ g_{12}\\ g_{23} \times g_{12}\end{matrix}$ &$\begin{matrix}&&id_{P_2}\\ f_{12} &&&& g_{23} \\ &&g_{12} \times f_{12}=f_{23} \times g_{23}\end{matrix}$ &$\begin{matrix}&id_{P_3}\\ &f_{23}\\ &f_{12} \times f_{23} \end{matrix}$ \\
\hline
 
$\Sigma'$ &$\begin{matrix} id_{P_1}\\ f_{12}\\ f_{23}*f_{12}\end{matrix}$ &$\begin{matrix}&id_{P_2}\\ \tau^-g_{12} && \hspace{-0.4cm} f_{23} \\ &f_{12} * \tau^-g_{12}=\tau^-g_{23}*f_{23}\end{matrix}$ &$\begin{matrix}&id_{P_3}\\ &\tau^-g_{23}\\ &\tau^-g_{12}*\tau^-g_{23}\end{matrix}$ \\
\hline
\end{tabular}
\end{center}
\end{exm}


\end{document}